\documentclass[12pt]{amsart}
\usepackage{amssymb,latexsym}
\usepackage{enumerate}
\usepackage{float}

\usepackage{mathtools}
\usepackage{amsfonts}
\usepackage{amssymb}
\usepackage{amsmath}
\usepackage{graphicx}
\usepackage{graphicx,color}
\usepackage[normalem]{ulem}
\usepackage{todonotes}

\numberwithin{equation}{section}
\makeatletter
\@namedef{subjclassname@2010}{%
  \textup{2010} Mathematics Subject Classification}
\makeatother

\usepackage{amsmath,amstext,amsthm,amsxtra,mathtools,amssymb}

\usepackage[bookmarksopen,bookmarksdepth=3,colorlinks,citecolor=red,pagebackref,hypertexnames=true]{hyperref}
\usepackage[backrefs,msc-links,nobysame,initials,non-sorted-cites]{amsrefs}

\usepackage[backrefs,msc-links,nobysame,initials]{amsrefs}

\allowdisplaybreaks

\newtheorem{thm}{Theorem}
\newtheorem{cor}{Corollary}
\newtheorem{lem}{Lemma}
\theoremstyle{definition}

\theoremstyle{remark}
\newtheorem{problem}{Problem}
\newtheorem*{remark}{Remark}

\begin{document}

\title[centered translation invariant density bases]{Tauberian constants associated to centered translation invariant density bases}
\author{Paul Hagelstein}
\address{P. H.: Department of Mathematics, Baylor University, Waco, Texas 76798}
\email{\href{mailto:paul_hagelstein@baylor.edu}{paul\!\hspace{.018in}\_\,hagelstein@baylor.edu}}
\thanks{P. H. is partially supported by a grant from the Simons Foundation (\#521719 to Paul Hagelstein).}

\author{Ioannis Parissis}
\address{I.P.: Departamento de Matem\'aticas, Universidad del Pais Vasco, Aptdo. 644, 48080
Bilbao, Spain and Ikerbasque, Basque Foundation for Science, Bilbao, Spain}
\email{\href{mailto:ioannis.parissis@ehu.es}{ioannis.parissis@ehu.es}}
\thanks{I. P. is supported by grant  MTM2014-53850 of the Ministerio de Econom\'ia y Competitividad (Spain), grant IT-641-13 of the Basque Government, and IKERBASQUE}


\begin{abstract}
This paper provides a necessary and sufficient condition on Tauberian constants associated to a centered translation invariant differentiation basis so that the basis is a density basis.   More precisely, given $x \in \mathbb{R}^n$, let $\mathcal{B} = \cup_{x \in \mathbb{R}^n} \mathcal{B}(x)$ be a collection of bounded open sets in $\mathbb{R}^n$ containing $x$.   Suppose moreover that these collections are translation invariant in the sense that, for any two points $x$ and $y$ in $\mathbb{R}^n$ we have that $\mathcal{B}(x + y) = \{R + y : R \in \mathcal{B}(x)\}.$   Associated to these collections is a maximal operator $M_{\mathcal{B}}$ given by
$$M_{\mathcal{B}}f(x) \coloneqq \sup_{R \in \mathcal{B}(x)} \frac{1}{|R|} \int_R |f|\;.$$  The Tauberian constants $C_{\mathcal{B}}(\alpha)$ associated to $M_{\mathcal{B}}$ are given by
$$C_{\mathcal{B}}(\alpha) \coloneqq \sup_{E \subset \mathbb{R}^n \atop 0 < |E| < \infty} \frac{1}{|E|}|\{x \in \mathbb{R}^n :\, M_{\mathcal{B}}\chi_E(x) > \alpha\}|\;.$$
Given $0 < r < \infty$, we set $\mathcal{B}_r(x) \coloneqq \{R \in \mathcal{B}(x) : \textup{diam } R < r\}$, and let $\mathcal{B}_r \coloneqq \cup_{x \in \mathbb{R}^n} \mathcal{B}_r (x).$  We prove that $\mathcal{B}$ is a density basis if and only if, given $0 < \alpha < \infty$, there exists $ r = r(\alpha) >0$ such that $C_{\mathcal{B}_r}(\alpha) < \infty$.   Subsequently, we construct a centered translation invariant density basis $\mathcal{B} = \cup_{x \in \mathbb{R}^n} \mathcal{B}(x)$ such that there does not exist any $0 < r$ satisfying  $C_{\mathcal{B}_{r}}(\alpha) < \infty$ for all $0 < \alpha < 1$.   
\end{abstract}

\maketitle

\section{Introduction}
This paper is concerned with the classical topic of ascertaining whether or not a given collection of sets in $\mathbb{R}^n$ constitutes a density basis.   A quick review of basic terminology in this regard is called for.  Given a point $x \in \mathbb{R}^n$, let $\mathcal{B}(x)$ be a nonempty collection of bounded open sets in $\mathbb{R}^n$ that contain $x$. The collection $\mathcal{B} \coloneqq \cup_{x \in \mathbb{R}^n}\mathcal{B}(x)$ is said to be a \emph{differentiation basis} if for every $x$ the collection $\mathcal{B}(x)$ contains sets of arbitrarily small diameter. On the other hand, the collection $\mathcal{B}$ is said to be a \emph{density basis} if, for every measurable set $E \subset \mathbb{R}^n$, and for almost every $x \in \mathbb{R}^n$, we have that
$$\chi_E(x) = \lim_{j \rightarrow \infty}\frac{1}{|R_j|}\int_{R_j}\chi_E $$
holds for every sequence of sets $\{R_j\}$ in $\mathcal{B}(x)$ whose diameters tend to 0.  For example, if $\mathcal{B}(x)$ is the collection of balls centered at $x$, then $\mathcal{B}$ is a density basis.  On the other hand, if $n \geq 2$ and $\mathcal{B}(x)$ is the collection of all rectangles centered at $x$, then $\mathcal{B}$ is not a density basis.  Proofs of these results may be found in, for instance, \cite{SteinHA}.

\emph{Maximal operators and Tauberian constants} provide useful tools for determining whether or not certain types of differentiation bases are in fact density bases.  Given a differentiation basis $\mathcal{B} = \cup_{x \in \mathbb{R}^n} \mathcal{B}(x)$ of sets in $\mathbb{R}^n$, the maximal operator $M_\mathcal{B}$ is defined by
$$M_{\mathcal{B}}f(x) \coloneqq \sup_{x \in R \in \mathcal{B}(x)} \frac{1}{|R|}\int_R|f|\;.$$
 Given a constant $0 < \alpha < 1$, the associated Tauberian constant $C_\mathcal{B}(\alpha)$ is defined by
$$C_{\mathcal{B}}(\alpha) \coloneqq \sup_{E \subset \mathbb{R}^n \atop 0 < |E| < \infty} \frac{1}{|E|}|\{x \in \mathbb{R}^n : M_{\mathcal{B}}\chi_E(x) > \alpha\}|\;.$$

A basis of the above type is said to be ``centered'' and the associated maximal operator is  considered to be centered as well.  It is important to note that the term ``centered'' does not necessarily relate to a geometric center of sets, but rather to the existence of a distinguished point in each set. Now, a \emph{Busemann-Feller basis} $\mathcal{B}$ is a differentiation basis such that, if $x \in R \in \mathcal{B}(y)$ then $R \in \mathcal{B}(x)$.  Busemann-Feller bases are ``uncentered'', and the maximal operator $M_\mathcal{B}$ associated to a Busemann-Feller basis admits a straightforward presentation
     $$M_{\mathcal{B}}f(x) \coloneqq \sup_{x \in R \in \mathcal{B}}\frac{1}{|R|}\int_R |f|\;.$$    To illustrate this point, the uncentered Hardy-Littlewood maximal operator acting on a function $f$ at $x$ provides the supremum of the average of $f$ over all balls containing $x$, but the centered Hardy-Littlewood maximal operator only averages over balls \emph{centered} at $x$.  The uncentered Hardy-Littlewood maximal operator thus corresponds to a Busemann-Feller basis; the centered one does not.  We remark that the commonly used notation $M_\mathcal{B}f$ is somewhat ambiguous; by context one must recognize whether the maximal operator is associated to a centered or uncentered basis.

     A basis $\mathcal{B} = \cup_{x\in\mathbb{R}^n}\mathcal{B}(x)$ is said to be \emph{translation invariant} provided $R \in \mathcal{B}(0)$ if and only if $x + R \in \mathcal{B}(x)$.
     A Busemann-Feller basis $\mathcal{B}$ is said to be \emph{homothecy invariant} provided that $R \in \mathcal{B}$ implies that any translate or dilate of $R$ also lies in $\mathcal{B}$.  The following result is a cornerstone of the theory of differentiation of integrals that enables one to determine whether or not a homothecy invariant  Busemann-Feller basis is in fact a density basis.

\begin{thm}[Busemann and Feller, 1934]\label{t1}
Let $\mathcal{B}$ be a homothecy invariant Busemann-Feller differentiation basis of sets in $\mathbb{R}^n$.  Then $\mathcal{B}$ is a density basis if and only if $C_\mathcal{B}(\alpha) < \infty$ for every $0 < \alpha < 1$.
\end{thm}

The original proof of this theorem may be found in \cite{busemannfeller1934}; the interested reader is highly encouraged to read the presentation of this result in \cite{Gu}*{Chapter III}.

Somewhat surprisingly, at the present time we have found no satisfactory analogue of this theorem in the case that $\mathcal{B}$ is simply a \emph{Busemann-Feller translation invariant} basis.  This in fact is a highlighted problem in de Guzman's \emph{Differentiation of Integrals in} $\mathbb{R}^n$.   In considering this problem, we recognized that the analogue of this theorem in the context of centered translation invariant  bases was surprisingly also unknown.   The purpose of this paper is to provide an analogue of the Busemann-Feller Theorem in the context of centered translation invariant bases, and in that regard we prove the following.

\begin{thm}\label{t2}
Let $\mathcal{B} = \cup_{x \in \mathbb{R}^n}\mathcal{B}(x)$ be a centered translation invariant  differentiation basis.  Given $r > 0$ and $x \in \mathbb{R}^n$, let $\mathcal{B}_r(x) = \{R \in \mathcal{B}(x) : \textup{diam }R < r\}\;$ and set $\mathcal{B}_r = \cup_{x \in \mathbb{R}^n}\mathcal{B}_r(x)$.  Then $\mathcal{B}$ is a density basis if and only if, given $0 < \alpha < 1$, there exists $ r = r(\alpha) >0 $ such that $C_{\mathcal{B}_r}(\alpha) < \infty$.
\end{thm}

We remark that, as is explicitly indicated in Corollary \ref{t3} below, this theorem remains true in the context of uncentered (Busemann-Feller) bases.  It is a matter of considerable interest whether or not, for a translation invariant Busemann-Feller density basis $\mathcal{B}$, there exists a \emph{uniform} value of $r > 0$ such that $C_{\mathcal{B}_r}(\alpha) < \infty$ for all $0 < \alpha < 1$.  As we shall see in Theorem~\ref{t.nouniformr} below, this is not the case for centered bases and the analog of de Guzman's problem for centered translation invariant bases has a negative answer.

The second section of this paper will be devoted to a proof of this theorem.  The primary techniques of the proof will entail relatively standard geometric arguments in conjunction with techniques associated to Stein-Nikishin theory.  In the third section we will provide an example of a centered translation invariant density basis for which there does not exist a \emph{uniform} $r > 0$ such that $C_{\mathcal{B}_r}(\alpha) < \infty$ for every $0 < \alpha < 1$.  In this last section we will also provide suggested avenues for further research in this area.

\section{Proof of Theorem 2}
This section is devoted to a proof of Theorem \ref{t2}.  The arguments regarding necessity and sufficiency being quite different, we isolate them into separate lemmas.

\begin{lem}[Sufficiency]\label{l1}
Let $\mathcal{B} = \cup_{x \in \mathbb{R}^n}\mathcal{B}(x)$ be a centered translation invariant  differentiation basis.   Suppose that, given $0 < \alpha < 1$, there exists $r = r(\alpha) >0$ such that $C_{\mathcal{B}_r}(\alpha) < \infty$.  Then $\mathcal{B}$ is a density basis.

\begin{proof}
In Theorem 1.1 of Chapter III of \cite{Gu}, de Guzman proved that if $\mathcal{B}$ is a Busemann-Feller basis, then $\mathcal{B}$ is a density basis if and only if, for each $0 < \lambda < 1$, for each nested sequence $\{A_k\}$ of bounded measurable sets such that $|A_k| \rightarrow 0$ and for every nonincreasing sequence $\{r_k\}$ of positive real numbers such that $r_k \rightarrow 0$, we have
$$|\{x \in \mathbb{R}^n : \, M_{\mathcal{B}_{r_k}}\chi_{A_k}(x) > \lambda\}| \rightarrow 0\;.$$  A close inspection of his argument indicates that the proof does not rely at any point on $\mathcal{B}$ actually being a Busemann-Feller basis; the result holds for $\mathcal{B}$ being simply a differentiation basis.  As the hypotheses of the lemma guarantee that
$$|\{x \in \mathbb{R}^n :\, M_{\mathcal{B}_{r_k}}\chi_{A_k}(x) > \lambda\}| \leq C_{\mathcal{B}_r}(\lambda)|A_k|$$
provided $r_k$ is sufficiently small (in particular less than $r$), the desired result holds.
\end{proof}

\end{lem}

We now consider the issue of necessity.  Note that \emph{any} Busemann-Feller differentiation basis $\mathcal{B}$ is a density basis provided that, given $0 < \alpha <1$, $C_{\mathcal{B}_r}(\alpha) < \infty$ for sufficiently small $r$ as is indicated by the above argument.  (In particular, translation or dilation invariance of the basis plays no role.)  It is important to recognize that the converse is in general false.  For example, one could let, say, $\mathcal{B}$ be the Busemann-Feller basis consisting of sets in $\mathbb{R}$ that are either intervals or are of the form $I_1 \cup I_2$ where, for some $j \in \mathbb{N}$, both $I_1$ and $I_2$ are open intervals in $(2^j, 2^{j+1})$ such that $2^{-j} < |I_1| = 2^j|I_2|$.   This is a density basis (as can be seen by the Lebesgue Differentiation Theorem), although $C_{\mathcal{B}_r}(\alpha) = \infty$ for every $0 < \alpha < 1$ and for  every $r > 0$.  For a more geometrically motivated example of this type, one could enumerate the squares in $\mathbb{R}^2$ of the form $[s, s+1] \times [t,t+1]$ for $s,t \in \mathbb{Z}$ by $Q_j$, enumerate the rational numbers by ${q_j}$, and let $\mathcal{B}$ be the Busemann-Feller basis consisting of all rectangles of length less than 1 that have center in $Q_j$ with slope lying in the set $\{q_1, \ldots, q_j\}$.    Note, however, that both of these examples lack both translation and dilation invariance.  In Theorem \ref{t1}, Busemann and Feller showed that necessity of finite Tauberian constants holds for a density basis provided the basis is translation and dilation invariant; our task is harder here because we lack dilation invariance.

 The reader may also find the presence of the ``$r$'' term in the necessity condition to be strange at first.  However, one must recognize that, in the absence of scale invariance, one might have for a translation invariant density basis $\mathcal{B} = \cup \mathcal{B}(x)$  that $C_{\mathcal{B}_{s}}(\alpha) = \infty$ for all $0 < \alpha < 1$ and sufficiently large values of $s > 0$. For example, let $\mathcal{B}(x)$ be the collection of all open sets in $\mathbb{R}^2$ containing $x$ of diameter greater than $\frac{1}{10}$ together with the collection of all balls in $\mathbb{R}^2$ containing $x$.  Given $0 < \alpha < 1$, it is only when we restrict ourselves to average over sufficiently small sets that we have a finite Tauberian constant.

\begin{lem}[Necessity]\label{l2}
Let $\mathcal{B} = \cup_{x \in \mathbb{R}^n}\mathcal{B}(x)$ be a centered translation invariant  differentiation basis.  If $\mathcal{B}$ is a density basis, then,  given $0 < \alpha < 1$, there exists $r = r(\alpha)$ such that $C_{\mathcal{B}_r}(\alpha) < \infty$.
\end{lem}
\begin{proof}
We prove the contrapositive.   We remark that our strategy here, in particular in its use of the Borel-Cantelli Lemma, is strongly motivated by the paper \cite{SSeq} of E. M. Stein regarding limits of sequences of operators.  Suppose there existed $0 < \alpha < 1$ such that there did \emph{not} exist $0 < r$ such that $C_{\mathcal{B}_r}(\alpha) < \infty$.  Let $k$ be a positive integer.  For notational convenience, we set $M_k = M_{\mathcal{B}_{2^{-k}}}$.  Let $\ell$ also be a positive integer.   There exists a set $S_{k,\ell} \subset \mathbb{R}^n$ of finite measure such that $|\{x \in \mathbb{R}^n : M_k \chi_{S_{k,\ell}}(x) > \alpha\}| \geq  2^{\ell}|S_{k,\ell}|\;.$  Observe that, given a cube $Q \subset \mathbb{R}^n$ of measure 1, $M_{k}\chi_{S_{k,\ell} \cap 3Q} = M_k \chi_{S_{k,\ell}}$ on $Q$.  By the pigeonhole principle we recognize that there exists a cube $Q \subset \mathbb{R}^n$ of measure 1 such that $|\{x \in Q : M_k\chi_{S_{k,\ell}}(x) > \alpha]\}| \geq 3^{-n}2^{\ell}|S_k \cap 3Q|$, where $3Q$ denotes the concentric 3-fold dilate of $Q$.   Letting $E_{k,\ell} = S_{k,\ell} \cap 3Q$, we then have the existence of a set $E_{k,\ell} \subset \mathbb{R}^n$ such that
$$|E_{k,\ell}| \leq 3^n 2^{-\ell} \textup{ and } |\{x \in \mathbb{R}^n : M_k \chi_{E_{k,\ell}}(x) > \alpha\}| \geq 3^{-n}2^{\ell}|E_{k,\ell}|\;.$$   Hence there exists a sequence of sets $\{R_{k,j}\}_j$  such that
$$\sum_j |R_{k,j}| < \frac{1}{2^k} \; \textup{ although } \sum_j |\{x \in \mathbb{R}^n : M_k \chi_{R_{k,j}}(x) > \alpha\}| = \infty\;.$$    By the Borel-Cantelli Lemma, there exists a collection of translates $\{\tau_{k,j}\}_j$ such that a.e. point in $\mathbb{R}^n$ is contained in infinitely many sets of the form $\tau_{k,j} \{x \in \mathbb{R}^n : M_k \chi_{R_{k,j}}(x) > \alpha\}$.  Let now $E = \cup_{j,k} \tau_{k,j} R_{k,j}$.   We have that $|E| < \infty$, although for every $r > 0$ we have that
$M_{\mathcal{B}_r} \chi_E > \alpha$ a.e. on $\mathbb{R}^n$, since $M_{\mathcal{B}_r} \chi_E \geq M_k \chi_E$ for every $k$ such that $2^k < r$.  Accordingly, for a.e. $x \in \mathbb{R}^n$ we have that there exists a sequence $R_{x,j}$ of sets in $\mathcal{B}(x)$ of diameters tending to 0 such that $\frac{|E \cap R_{x,j}|}{|R_{x,j}|} > \alpha$.  Hence $\mathcal{B}$ cannot be a density basis.
\end{proof}

\begin{proof}[Proof of Theorem \ref{t2}]
This follows immediately from Lemmas \ref{l1} and \ref{l2}.
\end{proof}

\section{Further results and future directions}
One of the highlighted problems in de Guzman's \emph{Differentiation of Integrals in $\mathbb{R}^n$} is to find the appropriate analogue of Theorem \ref{t1} in the context of translation invariant Busemann-Feller bases.  As an immediate corollary of Theorem \ref{t2}, we have the following.

\begin{cor}\label{t3}
 Let $\mathcal{B}$ be a translation  invariant Busemann-Feller differentiation basis of sets in $\mathbb{R}^n$.  Then $\mathcal{B}$ is a density basis if and only if, given $0 < \alpha < 1$, there exists $ r = r(\alpha)>0$ such that $C_{\mathcal{B}_{r}}(\alpha) < \infty$.
\end{cor}

An intriguing issue is, if $\mathcal{B}$ is a translation invariant Busemann-Feller density basis, whether or not there exits a value of $r$ such that $C_{\mathcal{B}_r}(\alpha) < \infty$ for \emph{all} $0 < \alpha < 1$.  That this is so was in fact suggested by de Guzman in \cite{Gu}.  We highlight this issue as follows:

\begin{problem}\label{prob.BFunif}
Let $\mathcal{B}$ be a  translation invariant Busemann-Feller density basis.  Must there exist $r>0$ such that $C_{\mathcal{B}_r}(\alpha) < \infty$ for all $0 < \alpha < 1$ ?
\end{problem}

This appears to be a very difficult problem.  We are pleased to be able to observe, however, that the above conclusion is \emph{false} in the general case of $\mathcal{B} = \cup_{x \in \mathbb{R}^n} \mathcal{B}(x)$ being a centered translation invariant density basis.

\begin{thm}\label{t.nouniformr}
There exists a centered translation invariant density basis  $\mathcal{B} = \cup_{x \in \mathbb{R}^n} \mathcal{B}(x)$ such that there does not exists $0 < r$ satisfying the condition $C_{\mathcal{B}_r} (\alpha) < \infty$ for all $0 < \alpha < 1$.
\end{thm}
\begin{proof}
 For each positive integer $k$, let
$$\mathcal{B}_k(0) \coloneqq \{(-\frac{\delta}{2},\frac{\delta}{2}) \cup (s, s + 2^{-k}\delta) :\,  0 < \delta < 2^{-2^k} \,\,\textup{ and }\,\, s \in (2^{-k}, 2^{-k} + 2^{-2^k})\}\;$$
and let $\mathcal{B}_k(x) \coloneqq \{x + R : R \in \mathcal{B}_k(0)\}.$  Now, let $0 < \alpha < 1$ and suppose that $k$ is a positive integer such that $2^{-k} < \frac{\alpha}{2}$.  If $E$ is a measurable set in $\mathbb{R}$ and $M_{\mathcal{B}_k}(x)\chi_E(x) > \alpha$, then there exists a set $R \in \mathcal{B}_k$ of the form $R = R_1 \cup R_2$, where $R_1$ and $R_2$ are disjoint intervals with $|R_2| = 2^{-k}|R_1|$, $x \in R_1$, and
$$\frac{|E \cap (R_1 \cup R_2)|}{|R_1 \cup R_2|} > \alpha\;.$$
Note that
$$\frac{|E \cap R_1|}{|R_1|} \geq \left(\frac{|E \cap (R_1 \cup R_2)| - |R_2|}{|R_1 \cup R_2|}\right) \frac{|R_1 \cup R_2|}{|R_1|} \geq (\alpha - 2^{-k})(1 + 2^{-k}) > \frac{\alpha}{2}\;. $$
Accordingly, $x \in \{x \in \mathbb{R} : M_{\textup{HL}}\chi_E (x) > \frac{\alpha}{2}\}$, where $M_{\textup{HL}}$ denotes the uncentered Hardy-Littlewood maximal operator.  As $M_{\textup{HL}}$ has a weak-type $(1,1)$ bound of 3, we then have $$C_{B_k}(\alpha) \leq \frac{6}{\alpha}$$ provided $2^{-k} < \frac{\alpha}{2}$.  Accordingly, by Theorem \ref{t2} we have that, setting $\mathcal{B}(x) = \cup_k \mathcal{B}_{k}(x)$, $\mathcal{B} = \cup_{x \in \mathbb{R}} \mathcal{B}(x)$ is a density basis.  However, suppose $r > 0$.  Set $k$ such that $\frac{2^{-k}}{1 + 2^{-k}} > \alpha$.  Setting $E = [0, \delta]$ for very small $\delta$ (in particular, less than $2^{-2^k}$, we have that $M_{\mathcal{B}_k}\chi_E > \alpha$ on an interval of length $2^{-2^k}$.  As $\delta > 0$ is arbitrarily small, we see then that for that value of $k$ we have that $C_{\mathcal{B}_k}(\alpha) = \infty$.  As all the sets in $\mathcal{B}_k$ have diameter less that $2^{-k+1}$, we see there is not a uniform value of $r$ for which $C_{\mathcal{B}_r}(\alpha) < \infty$ for all $\alpha > 0$.
\end{proof}

\begin{remark}  To the best of our knowledge, the density basis $\mathcal{B}(x)= \cup_{x \in \mathbb{R}^n} \mathcal{B}(x)$ constructed in the proof of the above theorem is the first known example of a centered translation invariant density basis not associated to a maximal operator satisfying a weak type $(p,p)$ or other nontrivial weak type estimate.
\end{remark}

It is worthwhile to consider translation invariant Busemann-Feller bases $\mathcal{B}$ such that $C_{\mathcal{B}}(\alpha_0) = \infty$ although $C_{\mathcal{B}}(\alpha_1)< \infty$ for some $0 < \alpha_0 < \alpha_1$.  An example of such a basis is
$$\mathcal{B} \coloneqq \{R = I_1 \cup I_2 : I_1 \textup{ and } I_2 \textup{ are intervals in $\mathbb{R}^1$ and } |I_1| = |I_2|\}\;.$$

One can show that for this basis $\mathcal{B}$ we have $C_\mathcal{B}(\alpha) < \infty$ for every $\frac{1}{2} < \alpha$ although \mbox{$C_{\mathcal{B}}(1/2) = \infty$.}  (Applications associated to this basis may be found in \cite{bh, hs}.)   Interestingly enough, we have been unable to find any translation invariant basis $\mathcal{B}$ such that $C_\mathcal{B}(1/2)$ was finite such that $C_{\mathcal{B}}(\alpha)$ was not finite for \emph{all} $0 < \alpha < 1$.  We highlight this as follows:

\begin{problem}If $\mathcal{B}$ is a translation invariant basis Busemann-Feller basis such that $C_\mathcal{B}(1/2) < \infty$, must $C_{\mathcal{B}}(\alpha) < \infty$ for all $0 < \alpha < 1$?
\end{problem}

If the above were true, we would have an affirmative answer to Problem~\ref{prob.BFunif}.  The reasoning is as follows.  Let $\mathcal{B}$ be a translation invariant Busemann-Feller density basis.  By Theorem \ref{t3}, we have that $C_{\mathcal{B}_r}(1/3) < \infty$ for some $r > 0$.   Hence for that same value of $r$ we would necessarily have
$C_{\mathcal{B}_r}(\alpha) < \infty$ for all $0 < \alpha < 1$.


\end{document}